\documentclass[12pt]{article}
\usepackage{epsfig,psfrag,amsmath,amssymb,amsthm}
\usepackage{graphicx,amssymb,xcolor}
\usepackage[english]{babel}
\usepackage{enumerate}
\usepackage{amsfonts,amsmath}

\begin{document}
\newcommand {\TT}{\mathbb{T}}
\newcommand {\Ind}{\mathbbm{1}}
\newcommand{\bcc}{\color{magenta}}
\newcommand{\bco}{\color{olive}}
\newcommand{\bcb}{\color{blue}}
\newcommand {\cP}{\mbox{${\cal P}$}}
\newcommand {\cS}{\mbox{${\cal S}$}}
\newcommand {\cO}{\mbox{${\cal O}$}}
\newcommand {\w}{\mbox{${\omega}$}}
\newcommand {\ep}{\xi}
\newcommand{\half}{\frac{1}{2}}
\newcommand{\ti}[1]{\tilde{#1}}
\newtheorem{stat}{Statement}
\newtheorem{examp}[stat]{Example}
\newtheorem{assump}{Assumption}[section]
\newtheorem{decth}[stat]{Theorem}
\newtheorem{prop}[stat]{Proposition}
\newtheorem{cor}[stat]{Corollary}
\newtheorem{thm}[stat]{Theorem}
\newtheorem{lemma}[stat]{Lemma}
\newtheorem{remark}[stat]{Remark}
\newtheorem{def1}{Definition}[section]

\title{
Many Greedy Cleaners in a Poisson Environment}
\author{Sergey Foss\\Heriot-Watt University and \\ Novosibirsk State University\footnote{Email address: s.foss$@$hw.ac.uk. Research of S.~Foss was supported by EPFL visiting fellowship and 
by RSF research grant 17-11-01173}\\ \and Thomas Mountford\\ \'{E}cole Polytechnique F\'{e}d\'{e}ral de Lausanne\footnote{Email address: thomas.mountford$@$epfl.ch. 
}}


\maketitle

\begin{abstract}
We introduce a new ``greedy cleaning'' model where
a star-like state space (containing $N$ halflines connected by the origin) is covered by a homogeneous Poisson process of ``dust particles'', and $N^{\alpha}$  cleaners/workers proceed with cleaning in a ``greedy'' manner: each worker chooses the closest particle next.
Assuming $\alpha \in (0,1)$, we analyse the asymptotic behaviour of the workers, as $N\to\infty$. We show that eventually all of them escape to infinity and that the way how do they do it depends on the value of $\alpha$. 
\end{abstract}


\section{Introduction}

Greedy cleaning models are known for a long time, and there is an increasing  
interest to the area within the recent years. An overview on the topic
may be found in \cite{BFL}.
\vspace{0.2cm}

\noindent
The broadly known ``greedy cleaning'' problem may be presented as follows. There is given an unbounded connected closed subset ${\cal X}$ of Euclidean space, having infinite Lebesgue measure. At time $t=0$, countably many ``dust" particles are placed there, as points of a Poisson process that has a constant intensity with respect to a natural measure associated with the state space. A single
 worker starts from a fixed location and removes these particles one-by-one as follows. There is another independent rate-1 Poisson process on the ``time'' halfline and, for $i=1,2,\ldots$, the worker waits for the $i$'th ring for exponential-1
random time $\xi_i$, then moves/jumps to a new (the $i$'th) particle,
instantaneously removes it and stays at this location until the next ring. 
So, by time $S_n=\xi_1+\ldots + \xi_n$, the worker cleans out $n$ particles from the space. We are interested in the following ``greedy-type'' dynamics:
each time the worker chooses the closest particle next. 
Then the basic question is: whether will the space ${\cal X}$ be cleaned from all particles by time $\infty$ or not?
Since the total number of dust particles is denumerable, we may arbitrarily numerate them and denote by $T_k$ the time of particle $k$ removal, so $T_k=S_n$ for some $n$ if particle $k$ is eventually removed and 
$T_k=\infty$,  otherwise.
Then one can formulate a more precise question: what is the probability ${\mathbf P} (T_k<\infty, \ \mbox{for all} \ k)$
and when, in particular, this probability equals 0 or 1?
\vspace{0.2cm}

\noindent
Note that in the case of a single worker the distribution of inter-ring  times does not play any role and has been introduced for clarity of the process above. However, its exponentiality will play an important role later on in this paper. 
\vspace{0.2cm}

\noindent
The introduced model may be viewed as an example of a {\it deterministic} walk in a {\it random} environment, which may be opposed to {\it random} walks in either {\it deterministic} or {\it random} environments.
\vspace{0.2cm}

\noindent
It appears that the formulated problem is either relatively easy or very hard to solve, depending on the state space. The case of a real line, ${\cal X} = {\cal R}$ is easy, and one can use the 0-1-type laws to obtain that, for $X(t)$ being the location of the server at time $t$,
\begin{align*}
{\mathbf P}(X(t) \to\infty) = {\mathbf P}(X(t)\to -\infty) =1/2  \  (\mbox{due to the space symmetry}),
\end{align*}
for any initial $X(0)$, and, moreover, there is known the explicit distribution of 
${\mathbf P} (\inf_t X(t) \in \cdot \ | \  X(t)\to\infty)$ and of
related characteristics (see \cite{G1}). The model becomes more complex when the Poisson process on the line is space-inhomogeneous.
The case where its rate is a function of the distance from the origin has been analysed in \cite{GT}, where the authors showed that, depending on the rate properties, either the whole space is cleaned out from the dust or only a half of it. 
A similar problem has been analysed
in the case where the state space contains two parallel lines, see \cite{G2}.
In \cite{RST}, the authors analysed a version of the original model of a homogeneous Poisson process on the line, but assume in addition that  each dust point needs independently a random number of the worker's visits (either 1, with probability $p$, or 2, with probability $1-p$) to be cleaned out. If $p=1$,
then this is the original model, and only half a line is eventually cleaned out. However, as it is shown in \cite{RST}, for any value $p<1$,  the whole
line is eventually cleaned out from the dust.
\vspace{0.2cm}

\noindent
In the 2D case, an answer to the key question is unknown, for any reasonable variant of the state space -- either the whole plane ${\cal X}={\cal R}^2$ or a cone or
a slab. There have been several attempts to solve the problem, both analytically and/or numerically. In particular, A.~Holroyd \cite{H} has run a number of simulations that do not conclusively support either answer to the question. 
There are many more advanced problems, related to the trajectory of the worker and, in particular, to hitting
times of certain areas of the state space (see again \cite{BFL}). 
\vspace{0.2cm}

\noindent
There is a belief that, by analogy to a random walk,  in dimensions 3 or more it should typically be ${\mathbf P} (T_k<\infty \ \mbox{for all} \ k)=0$. However, we are unaware of any rigorous results here.
\vspace{0.2cm}

\noindent
There is another class of models (we call them ``greedy service''  models) which is close in spirit to our ``greedy cleaning'' models.  In a greedy service model, there are again (one or more )``greedy'' workers/servers that clean/serve the dust, and, in addition, there is an independent temporal-spatial Poisson (or, more generally, renewal) input process of dust particles. The main problems here are to establish stability of a model (in the case of a bounded state space) or to analyse the asymptotic trajectory(ies) of the worker(s) (in the case of an unbounded state space). The first stability conjecture for  continuous circle has been formulated in \cite{CG}, see further \cite{RFK} for a survey on the subject. A substantial progress in understanding discrete-space models has been obtained
in the 90's, see \cite{FL1}, \cite{Sch}, \cite{FL2}, \cite{KM} and references therein. Within the last decade, a number of new results have been obtained for continuous-space models, see \cite{R}, \cite{LU}, \cite{FRS}, and finally the long-standing initial 
conjecture from \cite{CG} has been proven in \cite{RS} in the Markovian case, while there are still open questions in a broader generality.
\vspace{0.2cm}

\noindent
One can place several workers to clean the space. Greedy service models with 2 and with 3 workers on the positive halfline
have been analysed by Schmidt \cite{S}. A model with any number of workers on the
positive halfline and its further extensions are considered in
\cite{FM}. 
\vspace{0.2cm}

\noindent
It appears that there is a significant progress in the studies of greedy
service mechanisms, while only a little has been found in the greedy cleaning models.
\vspace{0.2cm}

\noindent
This paper contributes to the latter direction.
We introduce a new greedy cleaning model with several workers and with a ``star-like'' state space, which
is a union of several halflines connected by the origin, and consider its asymptotic behaviour when both the number of workers and the number of halflines tend to infinity, with the latter to be significantly bigger than the former. 
We can prove that, with probability 1, each worker has eventually to choose a halfline on which they will stay.  It is also clear that if the number of halflines is very large compared to the number of workers then the chance of two workers having the same halfline is small, while if the number of workers is strictly greater than the number of halflines then the pigeon principle demands that at least one line is shared.  
So the question arises as to how many workers are needed, as a function of the number of halflines for there to be a good chance of ``doubling up".  We stress here that, as far as we can see, given $N$, the number of halflines, it is not clear that the probability of "doubling up" is monotone in the number of workers.  Some monotonicity question are discussed in the appendix.
This question can be extended to considering when it is possible for $k$ workers to have the same ultimate direction.
\vspace{0.2cm}

\noindent
A star-like state space is a very natural object. Such topological structures  have been proven to be historically efficient (``All roads  lead to Rome''!). They naturally appear  in many areas: in (tele)communication networks (like call centres, multilayer networks), biology, physics, etc.
\vspace{0.2cm}

\noindent
 The rest of the paper is organised as follows.
 In Section 2 we introduce the model and formulate our main results. Then in Section 3 we obtain a number of preliminary results for the model with many halflines. In Section 4, we deal with results on times and prove the first statement of Theorem \ref{th1}.
Then in Section 5 we prove the second statement of Theorem \ref{th1} and
in Section 6 the last statement. In Section 7 we comment on a possible proof of Theorem
\ref{th2}. Section 8 contains a number of results for an auxiliary model with a single halfline. 
 In Appendix, we collect a few monotonicity properties of our model.

\section{Model Description and Main Results}

We consider a star-like state space: there are $N$ halflines that start from the origin in different directions. So the state space is ${\mathcal X}=\cup_{i=1}^N {\cal R}^+_i$ where ${\cal R}_i^+$ is the $i$th
halfline, and all halflines have a single common point (the origin).
We equip $\mathcal{X}$ with the following $L_1$-type distance:  for two points from ${\mathcal X}$, with one at distance $x$ from the origin and the other at distance $y$, the distance $d$ between them
is equal to $|x-y|$ if they lie on the same halfline, and  $d=x+y$ if the points belong to different halflines.
\vspace{0.2cm}

\noindent
Each halfline ${\cal R}_i$ is covered by a homogeneous rate-1 Poisson process
of ``dust" particles ${\cal K}_i$, and these processes are mutually independent. So ${\cal K} = \cup_{i=1}^N {\cal K}_i$ is a homogeneous Poisson process on the state space ${\cal X}$.  
In the case where a specific halfline ${\cal R}_i^+$ is being discussed, we may tacitly identify it with ${\cal R}^+$, the real halfline and, in particular, identify the point on ${\cal R}_i^+$ at distance $y$ from the origin with the positive value $y$.
\vspace{0.2cm}

\noindent
Let $0< \alpha <1$. We defer discussions of $ \alpha \geq 1$ for a companion article. 
We assume there are $N^{\alpha}$ ``workers'' that clean
the space from the dust (to be more precise, $[N^{\alpha}]$ workers, where $[x]$ is the integer part of number $x$ but in the following where refering to integer quantities we will drop the square brackets). Initially (at time $0$) all the workers are 
 located 
at the origin. Each worker $j$ has his own clock process which is represented by an independent rate-1 Poisson process $M_j(t)$. When $j$th clock rings,
worker $j$ instantaneously jumps to the closest existing dust
particle, removes it and stays at that location until its clock rings again.  
\vspace{0.2cm}

\noindent
We consider asymptotic dynamic properties of the model, as $N$ grows to infinity.
\vspace{0.2cm}

\noindent
We use the following terminology:
a jump of a worker is an {\it advance} (along a halfline) if this is either
a jump from the origin to the closest dust particle or a jump
from one to another dust particle on the same halfline. Otherwise,
this is a {\it skip} (from one halfline to another), this is a jump from (a dust particle on) one halfline
to (a dust particle on) another. In the case of a skip, we may also say
that  a worker {\it returns to}, or {\it passes the origin}.
  \vspace{0.2cm}

\noindent
We enumerate the lines $l=1,2,\ldots,N$ and the workers
$w=1,2,\ldots,N^{\alpha}$ and let $l_w(t)=i$ if worker $w$ is located on line $i$
at time $t$ (we let $l_w(t)=0$ if the worker is still at the origin at time $t$) and $d_w(t)$ its distance
from the origin at time $t$. 
\vspace{0.2cm}

\noindent
Let $\rho_t$ be the distance from the origin to the closest existing dust particle at time $t$. Let
$\sigma_x$ be the first time when some worker, say $w$, skips from
one halfline to a dust point at another halfline at a distance at least
$x$ from the origin. Clearly, $\sigma_x \ge \sup \{t: \rho_t \le x\}$ a.s. 
\vspace{0.2cm}

\noindent
For $t\ge 0$, introduce event $A_{2}^{(t)} = \{ l_{w_1}(s) \ne l_{w_2}(s), \ \mbox{for all} \ 
w_1\ne w_2 \ \mbox{and all}$ $s\ge t \}$, so the event occurs if all workers are located at different lines at all times $s$ starting from time $t$. 
Then let  $A_2 = \cup_{t\ge 0} A_{2}^{(t)}$  be the event that eventually all workers stay at different halflines. Further, let
$\theta = \theta_N$ be the time instant of the very last skip (of any worker).
\vspace{0.2cm}

\noindent
Our main results are stated in the following two theorems.

\begin{thm}\label{th1}
(I) For any $\alpha \in (0,1)$ and $x>0$, 
\begin{align}\label{eq:111}
{\mathbf P} (\sigma_x = \infty ) \to 1, \ \mbox{as} \ N\to\infty.
\end{align}
Further, for any $\varepsilon >0$,
\begin{align}\label{eq:1111}
{\mathbf P} (\theta_N \le N^{1-\alpha+\varepsilon} \to 1, \ \mbox{as} \ N\to\infty.
\end{align}
(II) If $2/3 < \alpha < 1$, then 
\begin{align}\label{eq:112}
{\mathbf P}(A_{2})\to 0, \ \mbox{as} \ N\to\infty.
\end{align}
(III) If $0 < \alpha <2/3$, then
\begin{align}\label{eq:113}
{\mathbf P}(A_{2})\to 1, \ \mbox{as} \ N\to\infty.
\end{align}
\end{thm}
\vspace{0.2cm}

\noindent
The approach offered is sufficiently robust to permit generalization.  Let 
 $A_{m}^{(t)} = \{\mbox{for all} \ s\ge t \ \nexists \mbox{ distinct } w_1, w_2, \cdots w_m : l_{w_1}(s) = l_{w_2}(s) = \cdots = l_{w_m}(s) \
 \}$ and
 let  $A_m = \cup_{t\ge 0} A_{m}^{(t)}$ .  Following the line of the proof of Theorem \ref{th1},
 with natural minor changes, one can obtain
 
 \begin{thm}\label{th2}
(I) If $(2m-2)/(2m-1) < \alpha < 1$, then 
\begin{align}\label{eq:114}
{\mathbf P}(A_{m})\to 0, \ \mbox{as} \ N\to\infty.
\end{align}
(II) If $0 < \alpha <(2m-2)/(2m-1)$, then
\begin{align}\label{eq:115}
{\mathbf P}(A_{m})\to 1, \ \mbox{as} \ N\to\infty.
\end{align}
\end{thm}

\section{Preliminary results}

The purpose of this section is to provide elementary results on the environment of ``dust" particles.
\vspace{0.2cm}

\noindent
We introduce the filtration 
$\{\mathcal{G}_t\}_{t \geq 0}$ where 
\begin{align*}
\mathcal{G}_t = \sigma ((l_w(s),d_w(s)), s\le t, 1\le w\le N^{\alpha})
\end{align*} 
is the sigma-algebra generated 
by the moves of the $N^\alpha $ workers by time $t$ (we consider the process to be right continuous).  
\vspace{0.2cm}

\noindent
We now describe the conditional distribution of the Poisson processes on the halflines at times $t$ (or at stopping times $T$).  If a halfline $l={\cal R}_i$ is fixed, then we can classify the ``information"  (for $l$ given) by four types of jumps occurring at time $s$ from point $x$ to point $y$.  To this end we introduce a variable $D_\ell (s) $ defined according to the four cases as below:\\
1) a jump to $l$, this is either an advance along $l$ or a skip to $l$ from another halfline (this means that $y$ belongs to $l$ while $x$ may be either on $l$ or on any other halfline).  Here we get the information $D_l(s) = y$ (where $y$ identified with its distance from the origin), this means : ``no dust particles on $l$ within distance $y$ from the origin''. \\
2) an advance within $j \ne l$.  Here $D_l(s) \ = \ (y-2x)^+$ .\\
3) a skip from $x\in j \ne l $ to $y\in k \ne l$. Here $D_l(s) \ = \ y$ \\
4) a skip from $x\in l$ to $y\in j \ne l $.  Here $D_l(s) \ = \ y+2x$. 
\vspace{0.2cm}

\noindent
At time $t$, let $I_l(t) \ = \ max_{s \leq t} D_l(s) $, the noneffaced dust particles on the halflines conditional upon $\{I_t (l) \}_{l} $ are conditionally independent Poisson processes on the 
halflines with, for any halfline $l$,  rate one on $[I_l(t), \infty )$ and rate zero on $(0, I_l(t))$.
\vspace{0.2cm}

\noindent
In particular if $\tau $ is, say,  the $j$'th  jump time for a given worker, then the environment ``seen" from this worker on its halfline away from the origin, given $\mathcal{G}_\tau $, is simply a rate-one Poisson process.  (Similarly -- for other stopping times representing worker jumps).
\vspace{0.2cm}

\noindent
Note that the process
\begin{align*}
X(t):= \left( (l_w(t),d_w(t)), 1\le w \le N^{\alpha}; I_l(t), 1\le l \le N\right)
\end{align*}
is a continuous-time Markov process that possesses the strong Markov property.
\vspace{0.2cm}

\noindent
In this and subsequent sections, we will introduce various ``bad" events $F_i$.
Mostly, such an event can be described as $F_i = \{T_i \leq  r_i\}$ or $\{T_i < \infty \}$, for some fixed time $r_i$ and some stopping time $T_i$ (with respect to filtration $( \mathcal{G}_t)_{t \geq 0} )$.  When we will say (for time $t \ \leq r_i \ $) that $F_i $ ``is not happened by time $t$", we will mean that $T_i > t$.
\vspace{0.2cm}

\noindent
By the superposition theorem for Poisson processes, the initial distribution of the distances of dust particles  (from all halflines) to the origin 
constitutes a rate-$N$ Poisson process.  This high intensity process will have many useful law of large numbers properties.   We also wish to establish regularity properties holding on each of the $N$ halflines' environments.   In particular,  we wish to show that each halfline cannot have ``too many" dust particles close to the origin and that a certain large scale regularity may be assumed (see Proposition \ref{prop1} below).  
\vspace{0.2cm}

\noindent
According to the introduced dynamics, workers jump (either advance or skip) at Poisson times.  We can see that, with  a high probability, all workers  for a long time will only make skips after the first advance from the origin. The intuition behind this is that, for 
the first dust point on any halfline, the next dust point on this halfline is usually on the distance of order 1, 
while the closest point on the other halflines is of the order of two times the distance of the current dust particle from the origin, which will be small compared to the O(1) distance along it current halfline.  
However we will show that 
unless the dust environment is negligibly ``extreme'', a single isolated worker will stop skipping after having made $\log^3(N)$ consecutive advances.
\vspace{0.2cm}

\noindent
In our model with many workers acting on the same halflines things can be more complicated.
In principle a worker could travel far in a halfline but still return to the origin even if the dust environment is not extreme since a large number of other workers could move to this halfline.  As
a technical approach, we consider behaviour of workers until the random time when the position of the 
dust particle closest to the origin (over all halflines) becomes greater than 1.  In fact we will see in the following section that this time will be infinity with probability tending to one, as $N $ tends to infinity.  (See (i) of the statement of Theorem \ref{th1}.)
\vspace{0.2cm}

\noindent
Recall from Section 2 our definitions of $\rho_t $  and  
$\sigma_x$.  
 In particular, for times before $ \sigma_1 $, no worker will enter a halfline whose closest dust particle is at distance greater than one from the origin. The value of this is that,
as we already claimed, 
no worker will skip before time $ \sigma_1 $ after having made $\log^3N$ consecutive advances. The time $\sigma_1$ will be shown to be infinite (see Section 4), with probability tending to one, so this conclusion, holding up to time $ \sigma_1 $ will hold forever.
\vspace{0.2cm}

\noindent
The first two propositions concern the regularity of the dust particles on the halflines and the regularity of Poisson processes.  Since they rely on simple 
estimates on probabilities for extreme events for  Poisson processes on the halfline, we leave the proofs to the reader.
\begin{prop} \label{prop1}
The following events on dust environments for $N$ halflines have probability tending to zero as $N$ tends to infinity:\\
$ F_{0} := \{$initially, there exists a halfline ${\it l}$ so that, for some $y \geq \log^2N$, the number of dust points on ${\it l}$ within distance  $y$ from the origin, $M_l(y) $,  is either less than or equal to $y/2$
or greater than or equal to $ 2y\}$;\\
$F_{1} :=\{$initially, there exists a halfline ${\it l}$ with at least $\log N$ dust points within distance $2$ from the origin $\}$;
\\
$F_{2} :=\{$initially, there exists a halfline ${\it l}$ so that, for some $y \geq \log^{2} N$, the number of dust points in $]y, 5y/4[$ is outside of the interval $(y/8,y/2)$.
\end{prop}
\vspace{0.2cm}

\noindent
{\it Remark:} To reinforce our discussion of the way "event $A$ has not occurred by time $t$" is used, we say $F_0 $ has not occurred by time $t$ (where $t$ is possibly a stopping time), if there does not  exist an $l$ and  $y \geq \log^2N$ so that $N_l(t,y)$, the number of revealed points by time t of $l \cap (0,y),$ satisfies $N_l(t,y) \notin (y/2, 2y)$.
\vspace{0.2cm}

\noindent
For any given $N$, let $T^{(N)} = \{t_{1}^{(N)} < t_{2}^{(N)} < t_{3}^{(N)} < \ldots \}$ be the ordered distances from the origin to the initial dust particles on the state space. As we mentioned, they form a rate-$N$ Poisson process.

\begin{prop} \label{prop2}
For $T^{(N)}$ as above, the following events have probability tending to zero as $N$ tends to infinity:\\ 
$ F_3 :=\{$there exists $i \leq \frac{N}{\log N}$ so that 
$t_{i + N^{\alpha} }^{(N)}- t_{i}^{(N)}  \notin (\frac{1}{2 N^{1-\alpha}} , \  \frac{2}{N^{1-\alpha}} ) \}$.\\
$F_4 :=\{$there exists $i\in (N^{1/2}, N/\log N)$ so that $\vert \frac{Nt_i^{(N)}}{i} -1 \vert \ \geq \ 2/\log N\}.
$
\end{prop}
\vspace{0.2cm}

\noindent
The next result is simple to prove but is useful in that it ensures a large supply of halflines which will be first visited at reasonably predictable times.
\begin{prop}  \label{bandersnatch}
For $ 0 < k_2 < k_1 < \infty $, let $W_N$ be the number of halflines whose closest to the origin dust point (at time $0$) lies in the interval
\begin{align*}
(  e^{- k_1 \sqrt{\log N} }, e^{- k_2 \sqrt{\log N} } ).
\end{align*}
As $N$ tends to infinity, 
\begin{align*}
    \frac{W_N}{N e^{- k_2 \sqrt{\log N} } }\to 1
\end{align*}
in probability.
\end{prop}

\noindent
The next result addresses the worker process rather than merely the initial dust particle environment.  We wish to show that when the process is reasonably advanced a worker after a skip (occuring definitely before time $\sigma_1$) has a reasonable possibility of advancing $\log^3N$ times without skipping.
\begin{lemma} \label{escape}
\noindent Fix $k > 0$ and $ 0 < \varepsilon < k \wedge (1- \alpha ) / 2$.  Consider the worker/dust process on $N$
halflines and let  $t$ be a stopping time at which worker $w$ skips. 
 to a dust particle on the halfline denoted ${\it l}$ at the distance from the origin 
 $ x \in \left[ {e^{-k \sqrt{\log N}}},{e^{-(k - \varepsilon)\sqrt{\log N}}} \right]$. 

\noindent Let $A$ be the event that, after time $t$, the worker $w$ advances at least $\log^3N$ 
consequent jumps along ${\it l}$ without skipping and within $2\log^3N$ units of time.                                                                                                                                                                                                                                                                                                                                                                                                                                                                                                                                           
and that all other workers, who were on ${\it l}$ at the moment $t$ of the worker's arrival (call them ``{\it black}'' workers), skip for other halflines without having advanced along ${\it l}$ by time $t+log^2(N)$ and no further workers enter ${\it l}$.  Let $A(\sigma_1) $ be the event that event $A$ has not been ruled out by time $\sigma_1 $,  which means that the selected worker $w$ has not skipped up to this time and if it has jumped  $\log^3N$ times then it has done so within time  $2 \log^3N$  and every ``black'' worker is either continuing to remove its current dust particle or has left after that. 
Then, on the event $D(t,w)$,  
$$
{\mathbf P}( A ( \sigma_{\bco 1}) \backslash (  F_1  \cup F_3  \cup F_4  ) | \mathcal{G}_t ) 
\ge \frac{1}{( \log N + 1)^{2}} \hspace{0.2cm} \left(\frac{1}{N}\right)^{\frac{k^2 (1+\varepsilon / k)}{2\log 2}} - 2e^{- \log^3N},
$$
for all $N$ large enough.  
\end{lemma}


\noindent
\begin{proof}
We first note that the fact that worker $w$ skips to a site distance $x<1$ from the origin implies that $t < \sigma_1 $.  Given our definition of event $A$ and of event $F_1$, we can suppose that the number of workers present on halfline $l$ at time $t$ 
is bounded by $\log N -1$ at all times less than  $\sigma_1$.
\vspace{0.2cm}

\noindent
We secondly note that, by the Markov property for Poisson processes, the conditional distribution of 
the dust points on $(x, \infty ) $ given $ \mathcal{G}_t $ is simply a rate-1 Poisson process.  Therefore we can apply Lemma \ref{lem1} and use simple Markov properties of Poisson processes to conclude that, given $ \mathcal{G}_t $ and for all $N$ large enough, the conditional probability  of the intersection of
the following two events $E_1$ and $E_2$ is at least $\exp\left(-(1+\varepsilon/200)\log^2x/ (2\log 2) \right)$. We define the events as follows: 
on the interval $(x, \infty ) $ of halfline ${\it l}$,\\ 
%
$E_1:=\{$there are a single dust particle on $[5x/3, 2x]$ and a single dust particle on $[25x/9, 10x/3]$ and there is 
no other dust particles on $]x, \frac{10x}{3}]\}$, and \\
$E_2:=\{$for a single half line model with only one worker at $x$, given these two dust particles the environment on $[10x/3, \infty )$
and with a single dust particle added at the origin, this worker 
will make 
 $\log^3N$ jumps to the right, without returning to the origin$\}$.


\vspace{0.2cm}
\noindent Given this dust environment event and assuming that event $F_1$ does not occur, the conditional probability that, among the workers present (worker $w$ and all ``black'' workers), the next two moves are made by the worker $w$ is 
at least $(\log N + 1)^{-2}$.  So the probability that the two events above occur is at least 
$( \log N + 1)^{-2} \exp\left(-(1+ \varepsilon/200)\log^2x/2\log 2\right) $.
\vspace{0.2cm}

\noindent
Now, unless one of the following events occur: \\
\indent
a) $w$ takes time greater than $2\log^3N $ to make $\log^3N$ moves or\\
\indent
b) the number of moves by any worker in the next $2\log^3N$ time units exceeds $4\log^3N$ or\\
\indent
c) either $F_3$ or $F_4$ is violated,\\

\noindent
we have that $\rho_s$, the position of the closest dust particle to the origin does not change by more than $x/3 $ in the time it takes for $w$ to make $\log^3N$ jumps and so event $A$ occurs.
\end{proof}
\vspace{0.2cm}

\noindent
The above result will now  be used to ensure that for any $k <   \sqrt {2\log 2(1-\alpha )} $, there is an ``appropriate" chance that 
a worker skipping into a new halfline at distance $x$ greater than $e^{-k \sqrt{\log N}}$ will then always advance along it until time $\sigma_1 $.

\begin{lemma} \label{happytostay}
\noindent Fix a stopping time $T$ and worker $w$.  Let $\lambda \ = \ \lambda (T, w) $ be the first time after $T$ that $w$ makes $\log^3(N) $ consecutive advances without skipping.
There exists universal strictly positive $c_0$ so that
\begin{align*}
{\mathbf P}(w \ \mbox{skips in} \ (\lambda, \sigma_1) \backslash ( F_1 \cup F_2))\ \leq \ 
 e^{-c_{0} \log^2N}
\end{align*}
 (We interpret the event to be empty if $\sigma_1 \leq \ \lambda $).
\end{lemma}
\noindent
{\it Remark:} Given the right-hand side upper bound, the lemma implies that the conclusion will hold for all workers, outside a set of small probability. 

\begin{proof}
We let $ {\it l} $ denote $ \ {\it l}_w(\lambda) $ and  $ x_0 $ denote $d_w(\lambda) $.
\vspace{0.2cm}

\noindent
Under $F_2^c$ the distance $x_0 $ must be at least $\log^2N + \frac{\log^3N- 2\log^2N}{4} \ \geq \ \frac{\log^3N}{8}$ for $N$ large.  
Note that  on $F_1^c$ from $\lambda $ up until time $\sigma _ 1 $ no further workers may enter ${\it l}$ and the number of other workers on the halfline is less than or equal to $\log N-1$.
\vspace{0.2cm}

\noindent
For $i \geq 1$, we define  $x_i \ = \ x_0 (\frac54 )^i $. We may assume, without loss of generality, that there is no dustpoints at any of these points.  We write $J_i$ for the interval
$(x_i, x_{i+1})$ and note that, on the complement of $F_2 $, the number of dustpoints in $J_i $ must exceed $x_0 (5/4)^i/8$, for all $i$.  We write $A_i $ for the event that $w$ cleans a point in $J_i $ before skipping and observe that, conditional upon event $A_i $, each dustparticle in $J_{i+1} $ has conditional probability (given $A_i$ and the cleaner of preceding particles) of being cleaned
greater than $\frac{1}{\log N }$.  Thus we have 
$$
P(A_{i+1}^c | A_1\cap A_2 \cdots \cap A_i \cap F_1^c \cap F_2^c) \ \leq \ \left(1 - \frac{1}{\log N} \right)^{x_0(5/4)^{i+1}/8}. 
$$
The result follows.
\end{proof}

\section{Results on Times}

In what follows, we say that, from time $t$ on, a worker $w$
{\it escapes}, or {\it escapes to infinity} (along a certain halfline $l$),
if the worker only advances along $l$ within time interval
$(t, \infty)$. Then the {\it escape time} is the time of the last skip of
the worker (or time 0 if there is no skip at all). 

\noindent
A major part of our work is understanding process $(  \rho_s : s \ \geq \ 0 )$.  We first give an easy upper bound.

\begin{lemma} \label{easy}
Fix $k \ \in (0, \infty ) $ and let $s_0 \ = \ N^{1-  \alpha } e^{- k \sqrt { \log(N)}}  $.  For any $ \delta > 0 $
$$
P( \rho_{s_0} \
 \geq \ e^{- (k - \delta)  \sqrt {\log(N)}}  \backslash F_4 ) \quad < \ e ^{-N^{(1- \alpha )/2 }},
$$
for $N$ large.
\end{lemma}

\begin{proof}
Using the Chernoff bound, we may conclude that, by time  $s_0$, outside  probability
bounded above by 
$N^\alpha e^{-bs_0}$, each worker has made less 
than $2s_0$ jumps,
for some universal strictly positive $b$. 

\noindent
Thus the number of jumps by time $s_0$ made by $\rho _. $, the position of the closest dust particle to the origin, is bounded above by 
$ 2 \exp \left(-k\sqrt{\log N}\right)$, outside the small probability given above. 
We have $\rho_{s_0} \  $ is not bigger than the $2 \exp \left(-k\sqrt{\log N}\right) $'th smallest value of the rate-$N$ Poisson process of dustparticles.
 Thus, outside of event $F_4$, the desired conclusion holds.

\end{proof}
\noindent
This proof (and similar ``upper bound" results for $\rho_. $) simply rely on large deviations bounds for Poisson processes: large positive deviations of the worker jump process, lower large deviation bounds for the dustparticle Poisson process.
The ``lower bounds" for $\rho_. $) are not so immediate: indeed they are not true for $k <  c :=  \sqrt{2(1- \alpha)\log 2}$ as we will see in this section.
The basic problem is that while we must have a good number of worker jumps not all of them are skips (which necessarily result in a change in $\rho_.$).
\noindent
We now get some simple bounds on the jump rates for $\rho_. .$

\noindent
We will use the results in Section 8 for workers on a single halfline to analyze the evolution of our model.

\noindent
We introduce a comparison model.  Let $(X_t^x: t \geq 0 ) $ be the position at time $t$ of a worker $w$ on $[0 , \infty )$ starting at $x$, where dustparticles are placed at $0$ and on $(x, \infty ) $ according to a rate $1$ Poisson process.   The position $0$ is taken to be absorbing.  This process is useful to compare with our multi worker process on $\mathcal{X}$. Here is a simple observation that follows from natural monotonicity properties (and does not need a proof). 
\begin{prop} \label{coup}
Consider a worker $w$ who skips at time $t$ to halfline $l$ in $\mathcal{X}$ at distance $y$ less than $x$ from the origin.  At time $t$  we generate a process $(X_s^{3x}: s \geq 0 ) $ by using the Poisson process of jump times for $w$ (shifted by $t$) as the jump times for $X$ and the dust particle environment on $(3x, \infty )$ for $X$ to be the dustparticles environment on $(y,\infty ) $ (on $l$ shifted by $3x-y$).  Then with this coupling, if $ X^{3x} _s = 0 $ for $s \leq \rho_{2x} -t$, then $w$ has skipped by time $s${\bco .}

\end{prop} 

\noindent
We now use Lemma \ref{lem1}  from the final Section to show the following

\begin{lemma} \label{half}
 Fix  $\delta > 0 $.  Assume that worker $w$ skips to a halfline $l$ at random time $t \ \leq \ \sigma_ {e^{- k \sqrt {\log(N)}}. }$ Let $A$ be the event that  $w$ advances to distance $1/100$ from $0$ before skipping. Then $P(A \backslash F_1) $ is bounded  above by
$
N^{- (k^2{\bco - }\delta)/(2\log(2))},
$
for all $N$ sufficiently large.
\end{lemma}

\begin{proof}
We apply Proposition \ref{coup} with $x \ = \ e^{- k \sqrt {\log(N)}}$.  We immediately have that the probability that $w$ advances to $1/100 $  before skipping
and before $F_1 $ does not occur is less than
$$
P ( X^{3 e^{- k \sqrt {\log(N)}}} \mbox{ reaches } 1/100  ) \ + \
P( \sigma_ {2e^{- k \sqrt {\log(N)}}} -  \sigma_ {e^{- k \sqrt {\log(N)}}} > Z)
$$
where $Z$ is the time for $w $ to make $\log(N)$ additional  jumps after random time $t$.  
The first bound is less than $ N^{- (k^2{\bco - }\delta)/(2\log(2))} /2 $ for $N$ large by Lemma \ref{lem1} and the second is of smaller order as $N$ tends to infinity, 
by standard Poisson computations. 
\end{proof}
\vspace{0.3cm}

\noindent
As noted, the conditional rates of the Poisson processes are known.  Though $\rho_t$ is not adapted to our filtration, the times of jumps of this process are adapted: there is a jump (that is a change)
in $\rho_t $ if  at time $t$ a worker skips to another halfline.
The jump rate for our process at time 
$t$ given the filtration $ \{ \mathcal{G}_s\}_{s \geq 0}$,
$$
\lim_{h \rightarrow 0} P \left( \rho_{t+h} \ \ne \ \rho_t  | \mathcal{G}_t \right) /h, 
$$
remains constant over time intervals free of worker jumps and on such intervals is at least the probability that the next jump is a skip (summed over all $N^ \alpha $ workers).
We similarly define the skipping rate to be \\
$$ 
\lim_{h \rightarrow 0} P \left(\mbox{ a skip occurs in time interval } (t,t+h) | \mathcal{G}_t \right) /h .
$$
\\
\begin{lemma} \label{goodrate}
Let event $F_5$ be defined by 
$$
F_5:=\{ \exists s \leq \frac{N^{1- \alpha}}{e^{- (c + \frac{2\varepsilon}{3})\sqrt{\log N}}} \
\mbox{so that the total skipping rate}  \mbox{ is less than } 9N^  {\alpha }/10 
\}.
$$
Then 
${\mathbf P}(F_5)  \to  0 $ as $N$ tends to infinity,
\end{lemma}

\noindent
{\it Remark:} The $9 / 10 $ bound could be improved but is sufficient for our purposes.
\begin{proof}
As noted above, jumps of workers through the origin necessarily change $\rho_.$ so we only need to analyse workers on a halfline.
It is easy to see that if a worker (at $x$ on halfline $l$) is at distance less than $1/100 $ from the origin then, conditional upon jumping, they will skip unless either there is a dustparticle in $(x,x+ 2/100) $ on $l$ or $\rho _t  \ \geq \ 1/100 $.  The latter is contained in the event
\begin{align*}
\{ \rho_{  \frac{N^{1- \alpha}}{\exp ({(c + \frac{2\varepsilon}{3})\sqrt{\log N}})}} \ \geq \ 1/100 \} 
\end{align*}
which, by Lemma \ref{easy}, has probability tending to zero. So it remains to show that the number of workers that reach distance $1/100 $ from the origin before time  $ \frac{N^{1- \alpha}}{e^{{\bco - }(c + \frac{2\varepsilon}{3})\sqrt{\log N}}} $ is small compared to $N^ \alpha $.  Fix a worker $w$.  By the Markov property, each time they skip to a point on a halfline  at distance less than 
$e^{{\bco -}(c + \frac{\varepsilon}{2})\sqrt{\log N}}$ from the origin, their chance of advancing along the halfline to distance $1/100 $ before skipping is, by Lemma \ref{half}, less than $N^{-(c + \frac{\varepsilon}{3})^2/ 2\log(2)} < N^{-(1- \alpha + b \varepsilon)  }$ (for $ b \ = \ \frac{2}{3} \sqrt{ \frac{1- \alpha}{2 \log(2)}}$ and $N$ large).  Thus the expected number of workers to reach distance $1/100 $ from the origin by time $\frac{N^{1- \alpha}}{e^{{\bco -}(c + \frac{2\varepsilon}{3})\sqrt{\log N}}}$ is less than $N^{-\varepsilon b }$, and the result follows.
\end{proof}
\vspace{0.5cm}
\noindent 
This bound and basic Poisson process bounds immediately yield 
\begin{cor} \label{goodrate2}
Let $a\ge 0, R\ge 1$, $b=a+R$ and let $N(a,b)$ be the number of jumps of $\rho_s$ in time interval $[a,b]$.
There exists an event $B_N$ of probability tending to one as $N$ tends to infinity
so that, for every $1 \leq R \leq N^{1- \alpha} $, the 
 conditional probability ${\mathbf P}(C(R)|B^c_N)$ of the 
event
\begin{align*}
C(R)  :=  \quad \{  \exists a : \  \rho_b \leq \frac{1}{e^{{\bco -}(c + \frac{\varepsilon}{2})\sqrt{\log N}}} \ \mbox{and} \ 
\vert N([a,b]) -RN^ \alpha \vert > RN ^{\alpha} / 5\}
\end{align*}
is bounded
above by 
$
 N/R e ^{-kRN^\alpha },
$
for some universal $k > 0$.
\end{cor}
\noindent
Another corollary is
\begin{cor} \label{llsigma}
For any $ \varepsilon > 0$, 
$$
{\mathbf P} \left(\sigma_{e^{-(c+\varepsilon )  \sqrt{\log N}}} \ < \  2N^{1- \alpha}e^{-(c+\varepsilon )  \sqrt{\log N}}
\right) \to 1, 
$$ 
as $N$ tends to infinity. 
\end{cor}
\begin{proof}
We simply
note that the  complement to the event $\sigma_{e^{-(c+\varepsilon )  \sqrt{\log N}}} \ < \  2N^{1- \alpha}e^{-(c+\varepsilon)}$ 
 is a subset of a union of three events, 
$F_5$, $F_6$ and $F_7$, where  $F_5$ was defined above, 
\begin{align*}
F_6 := \{ & \mbox{the number of all dust points within distance} \  e^{-(c+\varepsilon )  \sqrt{\log N}}  \\
& \mbox{ from the origin is greater than} \    
 3Ne^{-(c+\varepsilon )  \sqrt{\log N}} / 2 \}
\end{align*} 
and 
\begin{align*}
 F_7 := \{ & \mbox{the number of skips by time  } \\
& 
2Ne^{-(c+\varepsilon )  \sqrt{ \log N}}   \mbox{ is less than}  \ 3Ne^{-(c-\varepsilon )  \sqrt{\log N}} / 2\},
\end{align*}
and the probability of each of them tends to $0$ when $N$ grows,  that of $F_5 $ is simply Lemma \ref{goodrate}, $F_6$ involves simple Poisson process bounds and that for$F_7$
is a consequence of lemma \ref{goodrate}.
\end{proof}

\noindent
Our objective in the remainder of this section is to show that, in a crude sense, the dominant part of the workers escapes occur ``around" time $N^{1- \alpha} \exp (-c \sqrt{\log N })$ and that, for any $\varepsilon \in (0,c),$ we get 
$\sigma_{\exp (-(c-\varepsilon )  \sqrt{ \log N })} = \infty $ with probability tending to one, as 
$N \rightarrow \infty$.
\vspace{0.2cm}

\noindent
We now address upper bounds on times beyond which all workers do not return to the origin (i.e. only advance along halflines).\vspace{0.2cm}

\begin{thm} \label{propsigma}
\noindent For the worker/dust process on $N$ halflines, for each $ \varepsilon \ > \ 0 $, 
$$
 P ( \sigma_ {e^{-(c- \varepsilon) \sqrt N }} \ < \ \infty ) \ \rightarrow \ 0
$$
as $N$ tends to infinity, where as before $ c =  \sqrt{2(1- \alpha)\log 2} $.

\end{thm}

\noindent
{\it Remark: } This implies that the first statement of Theorem \ref{th1} holds.  In particular $P(\sigma_1 < \infty ) $ tends to zero as $N$ tends to infinity.
Thus all the preceding statements involving limiting probability of behaviour before time $\sigma_1 $ become strengthened to results over all time.
\vspace{0.2cm}

\begin{proof}
We fix $\varepsilon > 0$ small compared to $c$
and define two stopping times for our process:
$$ 
T_{1} = \sigma_{ \exp (- (c - \frac{\varepsilon}{8}) \sqrt{\log N}) } \ 
\mbox{and}
\
T_{2} = \sigma_{ \exp (- (c - \frac{\varepsilon}{4}) \sqrt{\log N})}. \ 
$$
We wish to show that $T_2 $ equals infinity with probability tending to one.  
In fact both stopping times can well be infinity and in fact our proof (and the arbitrariness of $\varepsilon $) will show that with probability tending to one as $N$ tends to infinity, this is so.
 Our first point is that on  $T_1 < \infty , \ T_2 - T_1 \geq   N^{1 - \alpha } e^{- (c - \frac{\varepsilon}{4}) \sqrt{\log N}} / 3$ 
outside an event of probability tending to zero as $N$ tends to infinity.  This follows from Proposition \ref{bandersnatch}  (applied with $k_2 = \ (c - \frac{\varepsilon }{4}) $ and $k_1 =\ (c - \frac{\varepsilon}{8}) $) and the fact that the corresponding dust particles are removed at rate bounded by $N^{\alpha } $. 
Let $A_1 $ be the (bad) event that this lower bound does not hold.
\vspace{0.2cm}

\noindent
 Again, on the event $\{T_1$ is finite$\}$, we have 
by elementary large deviations bounds on rate-1 Poisson processes that (outside an event of probability tending to zero as $N$ tends to infinity) every worker will make 
at least $    N^{1 - \alpha } e^{- (c - \frac{\varepsilon}{4}) \sqrt{\log N}} / 8 $ jumps in time interval $(T_1, T_1 + N^{1 - \alpha } e^{- (c - \frac{\varepsilon}{4}) \sqrt{\log N}} / 6)$.  The event that these laws of large numbers are not respected will be denoted by $A_2$.
\vspace{0.2cm}

\noindent
A third bad event, $A_3$, is that for at least one of the $N^ \alpha $ workers, say worker $w$, skips in time interval $(\lambda(T_1), \sigma_1) $ (where the stopping time $\lambda(T_1)$ is understood to be specific to worker $w$).  By Lemma \ref{happytostay} and the following Remark, this has probability tending to zero as $N$ becomes large.
\vspace{0.2cm}

\noindent
Finally let $A_4 $ be the event that, for some worker $w$, $\lambda^\prime (T_1) $  is more than $T_1 + N^{1 - \alpha } e^{- (c - \frac{\varepsilon}{4}) \sqrt{\log N}} / 6$
where $\lambda^\prime (T_1) $ is the time of the last skip preceding  $\lambda (T_1) $ as defined in Lemma \ref{happytostay}.  If $A_2 $ does not occur this would imply that the worker made
$$
N^{1 - \alpha } e^{- (c - \frac{\varepsilon}{4}) \sqrt{\log N}} / (8\log^3N) 
$$
skips before time  $T_1 + N^{1 - \alpha } e^{- (c - \frac{\varepsilon}{4}) \sqrt{\log N}} / 6$ without one of the skip times being  $\lambda^\prime (T_1)$.
By Lemma \ref{escape} and the Markov property, $P(A_4)$ tends to zero as $N$ tends to infinity.
\vspace{0.2cm}

\noindent
The result now follows by noting that on the intersection of the complements of the $A_i$ we have, by definition of $A_3$, 
that every worker does not skip on interval $(\lambda(T_1), \sigma_1) $
and hence not on $(\lambda^ \prime(T_1), \sigma_1) $.  Furthermore by the definition of $A_4$ and $A_2$ we also have that no worker will skip on interval $(T_1 + N^{1 - \alpha } e^{- (c - \frac{\varepsilon}{4}) \sqrt{\log N}} / 6, \sigma _1).$ But this and $A_1$ imply together that $T_2 $ must be infinite.
\end{proof}
\vspace{0.2cm}

\noindent We now wish to relate this to actual times.  We note that, by Lemma \ref{escape}, for $\varepsilon > 0 $ and $N$ large,  after time
$\sigma_{\exp (-(c+ \varepsilon) \sqrt{ \log  N})}$
 every time a worker changes halfline (and thus moves to a closest dustpoint ) it has a chance $ \frac{1}{N^{(1 - \alpha +3d\varepsilon /2 )}} $ of 
 coming to a halfline  which it will  not leave before time $\sigma_1$.
\vspace{0.2cm}

\noindent The chance that a worker can make $N^{1 - \alpha + 3d\varepsilon }$ visits without the above event occuring 
does not exceed $e^{- N^{\varepsilon }}$. So with probability tending to zero as $N$ becomes large none of the $N^{\alpha}$ workers satisfies this.
\vspace{0.2cm}

\noindent Let $s_w$ be the number of skips of worker $w$ by time $\sigma_1$.
Let
event $F_8^c$ be that, for each worker
and  for each of his first $\min (s_w,N^{1+ 2d\varepsilon }) $ skips to another halfline, all the numbers of advances between consecutive 
skips do not exceed $\log^3 N $.
\vspace{0.2cm}

\noindent Given Lemma \ref{escape}, we can say that ${\mathbf P}(F_8)$ tends to zero as $N $ becomes large. But unless event $F_1$ occurs, each jump to a new halfline after time
$\sigma_{\exp (-(c+ \varepsilon) \sqrt{\log N})}$ gives a chance at least 
\begin{align*}
\frac{K}{N^{1- \alpha +2 \varepsilon / d}\log^2N}
\end{align*}
of escape, for universal $K$.  
So the probability that there exists a worker $w$ which has not escaped by time 
$N^{1- \alpha +3 \varepsilon / d} \log^5 N$ is bounded by 
$$
{\mathbf P}(F_1) + {\mathbf P}(F_8) + \left(1- \frac{K}{N^{1- \alpha +2 \varepsilon / d}\log^2N}\right) ^{N^{(1- \alpha 
+3 \varepsilon / d)} \log^2 N},
$$
for some universal $K$. This implies statement \eqref{eq:1111} of Theorem \ref{th1}, since all three probabilities above tend to $0$ when $N$ grows. 

\section {Proof of existence of ``double" halflines}

In this section we prove Theorem \ref{th1} for the case of $ \alpha > 2/3$.   That is with probability tending to one as $N $ tends to infinity, there will be a halfline on which a pair of workers 
will eventually travel.    We fix $ \varepsilon < ( \alpha - 2/3) / 100 $.
\vspace{0.2cm}

\noindent
We note that, by Corollary \ref{llsigma}, for $ \varepsilon > 0 $ fixed with probability tending to one (as $N$ tends to infinity), $\sigma_{\exp(-(c +  \varepsilon/2)\sqrt{\log N})} < \infty $. 
Furthermore by Proposition \ref{bandersnatch}, with probability close to one, we have at least $N\exp(-(c +  \varepsilon/2)\sqrt{\log N})/2$ halflines having the property that their first dustparticle is (at time $0$) at a distance from the origin in the interval
$$
\left( \exp (-(c +  \varepsilon)\sqrt{\log N}), \exp (-(c +  \varepsilon/2)\sqrt{\log N}) \right).
$$
We denote by $V$ the set of such halflines and by $V(t)$ the subset of  halflines in $V$ which have not been visited by any worker by time $t$.
\vspace{0.2cm}

\noindent
So we may conclude that  if we define the stopping times $s_i $ recursively by $s_0 \ = \ \sigma_{\exp (-(c +  \varepsilon)\sqrt{\log N})} $ and
$s_{i+1} \ = \ \inf \{ t > s_i: \mbox{ a worker skips to a halfline in } V(s_i)\}$,
then, with probability tending to one as $N $ tends to infinity,
\begin{align*}
s_{N\exp (-(c +  \varepsilon/2)\sqrt{\log N})/2 } < \sigma_{\exp (-(c +  \varepsilon/2)\sqrt{\log N})}.
\end{align*} 
  We denote by ${\it l}_i $ and $x_i$ the halfline and dustparticle position associated with $s_i$.  
\vspace{0.2cm}

\noindent
The advantage of considering just the stopping times $s_i $ is that at time $s_i $ there will be a unique worker on ${\it l}_i $ and we will have, conditional upon $\mathcal{G}_{s_i}
$, a rate-1 Poisson dust environment on $(x_i, \infty )$.  
\vspace{0.2cm}

\noindent
We say that stopping time $s_i $ is {\it bold} if \\
(I) the position $x_i$ is within the distance 
\begin{align*}
[ e^{-(c + {\varepsilon})\sqrt{\log N}}, e^{-(c + \frac{\varepsilon}{2})\sqrt{\log N}} ]
\end{align*} 
from the origin;\\
(II) the second dustparticle $x_i(2)$ on $l_i$ is within distance $N^{- ( 1 - \alpha +\varepsilon) } $ of $x_i$;\\
(III) the third dustparticle $x_i(3)$ on $l_i$ is of distance at least $\frac{5x_i}{4}$ from the origin;\\
(IV) the environment on the line $l_i$ is ``$2\varepsilon $ good''.  That is: if we consider an auxiliary model of a single halfline $l_i$ with
dust particles located at $x_i$,$x_i(2)$ and $x_i(3)$, plus the given dustparticle environment within $(x_i(3),\infty)$ and an extra dust particle at the origin, and if we assume that there are only two workers, who are initially located  at $x_i$ and $x_i(2)$, then both workers have  chance at least $\frac{1}{N^{\varepsilon/2}}$ of escaping the origin. 
\vspace{0.2cm}

\noindent
It is clear from Lemma \ref{lem3} and basic properties of the Poisson process that the probability that $s_i$ is bold is at least $\frac{K}{N^{1- \alpha + \varepsilon}} \times \frac{1}{N^{2(1- \alpha + \varepsilon )}},$
for universal $K$.  
\vspace{0.2cm}

\noindent
We say that $s_i $ is a {\it success} if it is bold and, in addition, \\
(i) a second worker arrives at the second nearest dust point within time $N^{- \varepsilon / 2} $ and the first worker at $x_i$ does not move within this time interval and \\
(ii) the first and second workers do not pass the origin again, having made $\log N  $ jumps in next 
$N^\varepsilon /2$ units of time.\\

\noindent
As a useful comparison process for jump time $s_i $ if (i) above occurs at time $\tau_i $, then $Y( l_i, \tau_i ) $ will denote the $2$ worker process on $l_i $ 
with initial worker positions being those on $l_i $ at time $\tau_i $  with the given dustparticle environment on $l_i $ augmented by a dustparticle at the origin
and with the two jump Poisson processes given by the two workers on ${\it l}_i $ in the larger model.

\noindent
We see easily that the probability that $s_i $ is a success is greater than or equal to 
$\frac{c}{N^{3(1- \alpha + \varepsilon )}},$
for universal $c$.  Let $A_i $ be the event that $s_i $ is a success.  
The $A_i$ are (slightly) dependent, so we need to attend to some technical issues to complete the proof.
\vspace{0.2cm}

\noindent
Let us define event $F_{10} $ to be that for some Poisson jump time $t$ in interval $(\sigma_{e^{(c + {\varepsilon})\sqrt{\log N}}}, \sigma_{e^{(c + {\varepsilon} / 2  )\sqrt{\log N}}}) $ we have that there are less than  $\frac{N^{\alpha - \varepsilon/2}}{2} $ skips by workers in time interval $(t,t+ N^{- \varepsilon /2 } )$ or there are for the same interval more than $2N^{ \alpha - \varepsilon/2} $ worker jumps or event $F_1 $ occurs.  
We have by Corollary \ref{llsigma} and Lemma \ref{goodrate} and basic Poisson bounds that $P(F_{10}) $ tends to zero as $N$ tends to infinity.  
\vspace{0.2cm}

\noindent
It has to be noted, for $N$ large, that if $F_{10} $ does not occur then for each $s_i $ provided the associated worker does not move in the succeeding time interval of length $N^{\varepsilon / 2}$,
then condition (i) of success will automatically be satisfied and that (provided $F_2$ does not occur) then the required movement of the two workers on $l$ implies that no further workers arrive on ${\it l}$.  Thus if we consider $\tau_i $ to be the arrival of the second worker on ${\it l}$, then $s_i $ is a success if and only if process $Y({\it l}_i, \tau_i) $ satisfies corresponding conditions.
These considerations yield the following result which clearly implies
statement (iii) of Theorem \ref{th1}.
\vspace{0.2cm}

\noindent
Given these facts and results     we have
\begin{lemma}\label{lln}
There is a coupling between the worker/dust process and  a collection of indicator random variables $\{K_i\}_{ i \leq N\exp(-(c +  \varepsilon/2)\sqrt{\log N})/2} $ so that 
\begin{itemize}
\item
The $K_i $ are i.i.d. with $P(K_i = 1 ) = N^{-3(1- \alpha ) - 4 \varepsilon } $. 
\item
On event $B_N^  c $ where ${\mathbf P}(B_N) \ \rightarrow \ 0 $ as $N \rightarrow \infty $, we have $K_i = 1 \ \Rightarrow \ s_i $ is a success.
\end{itemize}
\end{lemma}
\noindent
From this the desired result is shown.

\section {Nonexistence of double escape}

\noindent To show that if $\alpha < \frac{2}{3}$, we do not obtain any halflines on which two workers escape, we choose $\varepsilon > 0$ so that
$
100 \varepsilon < 2/3 - \alpha.
$
\vspace{0.2cm}

\noindent
We will need the following result which has affinities with Proposition \ref{bandersnatch}

\begin{lemma} \label{regul}
Fix $ k > 0 $.  With probability tending to one as $N $ tends to infinity, for all $r \leq e^{-k \sqrt { \log (N)} } $,
$$
\sigma_{r + N^{-(1- \alpha )}} \ \geq \ \sigma_r \ + \  1/4.
$$ 
\end{lemma}
\begin{proof}
We fix $r \leq e^{-k \sqrt { \log (N)} } $.  The event \\
$B_r \ = \ \{ \sigma_{r + N^{-(1- \alpha )} / 2} \ \leq \ \sigma_r \ + \  1/4\}$ is contained in the union of events \\
$B_{r,1}=\{$ number of worker jumps between $\sigma_r $ and $\sigma_r + 1/4$ is at least $3N^ \alpha/7 \}$  and \\
$B_{r,2}=\{$ number of halflines so that the first dust particle is in $(r, r  + N^{-(1- \alpha )} / 2)$ is at most $3/7 N^ \alpha \}$.\\

\noindent
By elementary Poisson and Binomial tail probability bounds we find that $P(B_r ) \ \leq e^{-c_1 N^ \alpha } $, for some universal $c_1>0$ and uniformly in $r$, for $N $ large enough.  Taking $ r_ i = (i-1) N^{-(1 - \alpha)} / 10$ for $i = 1,2, \ldots, 10 N^{1 - \alpha }e^{-k \sqrt { \log (N)} } $, we get
$$
P( \cup _ i B_{r_i} ) \ \rightarrow \ 0 
$$
as $N$ tends to infinity, which gives the result.
\end{proof}
\noindent
For worker $u$ we write $D'(u)$ for the event that there exists another worker $v$ so that $u$ and $v$
escapes to infinity on the same halfline and that the final arrival time of $u$ to this halfline precedes that of $v$.  In general, other workers may also escape along the same halfline and their final arrival times may occur before that of $u$.  
For technical reasons we will work with event $D(u) \ = \ D'(u) \cap ( \cup _{i=1}^ {13} F_i)^c $ where $F_i $ are ``bad " events of probability tending to zero as $N $ becomes large.  Most of the $F_i $ have already been introduced, some remain to be defined.\\

\noindent
Let $F_{12} $ be the union of the following three events,
$\{ \sigma_{e^{-(c- \varepsilon ) \sqrt {\log N}}} \  < \ \infty \} $
and 
  $\{$there is a worker which has not escaped by time  $N^{1- \alpha + \varepsilon }\}$ and
 $\{$there is a worker making $2 N^{1- \alpha + \varepsilon } \mbox{ jumps by time } N^{1- \alpha + \varepsilon }\}$.

\noindent
By the first statement of Theorem \ref{th1} (proven at the end of Section 4) and Theorem \ref{propsigma} and elementary Poisson bounds, the probability $P(F_{12}) $ tends to zero as $N$ tends to infinity.

\noindent
Let $F_{13} $ the union of two events,
$\{$there is a worker $u$ so that for some jump (of $u$) while $t \leq N^{1- \alpha + \varepsilon }$, the worker makes less than $N^{\varepsilon / 2} /2$ jumps in the next $N^{\varepsilon / 2} $ time units $\}$ and
$\{$the conclusion of Lemma  \ref{regul} fails to hold$\}$.\\

\noindent
It is clear $P(F_{13}) \ \rightarrow \ 0$ as $N$ becomes large.
We wish to analyze $D(u)$.  We first note that $D'(u) \ = \ \cup_j D'(u,j) $ where $D'(u,j) $ is the event that after the $j$'th  jump of worker $u, \tau^u_j $, but before the $(j+1)$'st a different worker skips to the current halfline of $u$ and that thereafter the two workers do not skip.

\noindent
An advantage of working with the events  $D(u)$ is that since $D(u) \ \subset \  F_{12} ^c$ for each $u$
 
$$
D(u) \ = \ \cup_{j=1} ^ {2 N^{1- \alpha + \varepsilon} } D(u,j) 
$$ 
where $D(u,j) \ = \ D'(u,j) \cap D(u)$.  
Another advantage of using the events $D(u,j) $ is the following
\begin{lemma} \label{CJM}
Event $D(u,j)$ is contained in the event that after the $j$'th jump of $u$ there is a dustparticle on the same halfline within distance $N^{1- \alpha + \varepsilon} $ of $u$,
as $N$ becomes large.
\end{lemma}
\begin{proof}
If $F_{13} $ does not occur, then worker $u $ in the time interval  $( \tau^u_j , \tau^u_j +N^{ \varepsilon / 2  } )$  must make at least  $ N^{ \varepsilon / 2  } / 2 >> \ \log(N)$ jumps.  So  if $F_1 $ 
does not occur, then in this time interval 
either $u$ skips or all dustparticle up to distance $1$ are cleaned.  But, again if $F_{13} $ does not occur so the conclusions of Lemma \ref{regul} hold, then unless there is a dustparticle within distance $N^{1- \alpha + \varepsilon } $ of $u$ at time $\tau^u_j $, no particles can skip onto the halfline of u before $\sigma_1 $ which is infinity if $F_{12}$ does not occur.  

\end{proof}

\noindent
In what follows, we need the following notion. An array of points $\mathcal{P} $ on interval $(x, \infty ) $ is said to be {\it $k-(\delta)$blocking} if there exists $(y,2y] \ \subset \ ( \delta, 1) $ so that 
$ ( \mathcal{P} -x) \cap (y,2y]$ has fewer than $k$ points. 
Here $ \mathcal{P} -x $ is simply the translation  of $ \mathcal{P} $ for $x$ units to the left.
This is discussed further in Section 8.

\begin{lemma} \label{block}
If a worker $w$ skips to a halfline $l$ at random time $t \ \leq \ \sigma_ {e^{- (c- \varepsilon) \sqrt {\log(N)}}. }$ where $u$ is also located then if the 
dustparticle environment
to the right of $w$ is  $ 2-(2e^{- (c- \varepsilon) \sqrt {\log(N)}}$) blocking then it is not possible that two of the workers currently on the halfline  escapes on $l$ (without subsequent skipping)
before time $\sigma_{ e^{- (c- \varepsilon) \sqrt {\log(N)}}}$ .
\end{lemma}

\vspace{0.2cm}




\noindent
The above result, the Markov property, Corollary \ref{Willwriteit} (and the definition of event $D(u)$) immediately imply 
\begin{lemma} \label{probest}
For all $u$ and $ j \leq  \ 2N^{1- \alpha + \varepsilon
}$ and for all $N$ sufficiently large, 
\begin{align*}\ 
 {\mathbf P}(D(u,j) ) \ < \ N^{ -3(1- \alpha) + 3 \varepsilon }.
 \end{align*} 
  \end{lemma} 
\vspace{0.2cm}
\begin{proof}
This simply follows since by Lemmas \ref{CJM} and \ref{block}.  This gives that for $D(u,j) $ to occur there must be a dustparticle within $N^{1- \alpha + \varepsilon } $ of $u$ at time $\tau^u_j $ and that 
and the  dustparticle environment to the right of this dustparticle cannot be $2-$blocking.  The bound now follows from the Markov property and Corollary \ref{Willwriteit}
\end{proof}

\noindent
{\it Proof of Theorem \ref{th1} (III)}
\noindent
The  bound provided by Lemma \ref{probest} is sufficient to show part (III) of Theorem \ref{th1}. We have that the probability of two workers escaping on the same halfline is less than
$$
P (\cup_{i=1} ^{ 13} F_i ) \ + \ P ( \cup _ u D(u) ) \ =  P (\cup_{i=1} ^{ 13} F_i ) \ + \ P( \cup _ u \cup \cup_{j=1} ^ {2 N^{1 -\alpha + \varepsilon}} D(u,j) )
$$
$$
\leq \ P (\cup_{i=1} ^{ 13} F_i ) \ + \ N^ \alpha 2 N^{1 -\alpha + \varepsilon}N^{ -3(1- \alpha) + 3 \varepsilon }
$$

\noindent
which is $P (\cup_{i=1} ^{ 13} F_i ) \ +\ 2 N^{\alpha  -2(1- \alpha) + 4 \varepsilon }$ $= \ o(1) \ + \ 2N^{-(2-3 \alpha ) +4 \varepsilon} $ which tends to zero by our assumption on $ \varepsilon $.

\section{Comments on the Proof of Theorem \ref{th2}}

\noindent
The arguments given in the two preceding sections readily generalize to the case of more than two workers escaping together.  So we provide only a sketch of the proof of Theorem \ref{th2}.  
\vspace{0.2cm}

\noindent
To show that event $A_m $, as defined in Section 2, is likely for $ \frac{2m-2}{2m-1} \ < \ \alpha $ as $N $ becomes large, we fix $0 \ < \  \varepsilon \ \ll \ \alpha -  \frac{2m-2}{2m-1} $ and operate on 
times less than $e ^ {-(c+ \varepsilon ) \sqrt { \log (N) } }$ for which the evolution of $\rho _.$ is predictable and governed  by laws of large numbers.  Proposition \ref{bandersnatch} implies that we will be visiting many halflines for the first time
at a near deterministic rate.  We simply modify the definition of {\it bold} given in section 5: (II) is changed to require that beyond the first dustparticle there are $m-1$ dustparticles within distance $N^{1- \alpha - \varepsilon }$ and (IV)
is changed to require that the dustparticle environment
 is such that $m$ particles can escape with ``reasonable" probability. 
 \vspace{0.2cm}

\noindent
To show that for $\alpha < \ \frac{2m-2}{2m-1} $ event $A_m $ has small probability for $N$ large we first argue, as in Section 6, that this event is essentially the event that there is a halfline with a dustparticle close to the origin so that there are
$m-1 $ other dustparticles within  $N^{1- \alpha +\varepsilon } $ of the first (for $ 0 \ < \ \varepsilon \ \ll  \frac{2m-2}{2m-1} -  \alpha $) and that the environment is not $m$-blocking (rather than $2$-blocking).  Thereafter the argument is the same.

\section{Auxiliary results -- Cleaning process on the halfline}
\begin{lemma}\label{lem1} 

Consider an auxiliary model, with a single halfline and a single worker which is initial ly located at distance $x>0$ from the origin.
Assume that initially there are infinitely many dust particles on the halfline, with one locating at the origin and all the others at points of a rate-1 Poisson process on the set $(x,\infty )$. The worker always jumps to the closest existing dust particle and 
removes it. Let $D(x)$ be the event that the dust particle at the origin will be never removed and let $P_1(x)$ be its probability. Then
\begin{equation}\label{F1}
\lim_{x\to 0} \frac{\log P_1(x)}{\log^2 x} = -\frac{1}{2\log 2}.
\end{equation}
\end{lemma}
\noindent Here is an extension of Lemma \ref{lem1} onto the case of $k$ workers 
located at the same point. 

\begin{lemma}\label{lem2}
Consider a single halfline and $k$ workers located initially at the same point at distance $x>0$ from the origin.
Assume there are dust particles that are located at points of a rate-1 
Poisson process at $(x,\infty )$
and at the origin. Each worker has its own Poisson rate-1 clock, and the clocks ring independently of each other.
The workers always jump to the closest dust particle and 
remove it. Let $D_k(x)$ be the event that the dust particle at the origin will be never removed and $P_k(x)$ its probability. Then
\begin{equation}\label{F2}
\lim_{x\to 0} \frac{\log P_k(x)}{\log^2 x} = -\frac{k}{2\log 2}.
\end{equation}
\end{lemma}

\noindent
We need now the notation of $k-(\delta)$ {\it blocking array} that has been introduced
close to the end of Section 6.
The proof of Lemma \ref{lem2} immediately implies
\begin{cor} \label{Willwriteit}
For any $x > 0 $ and $\gamma \ > \ 0 $ and for all $\delta $ sufficiently small, 
the probability that a rate one Poisson process of points on $(x, \infty ) $ is not $k-( \delta)$blocking is less than
$$
e ^{- ( \log( \delta) ) ^ 2 k(1- \gamma) / 2 \log(2)}
$$
\end{cor}

\noindent
{\it Remark:} in the above result the variable $x$ plays no role, but it is formulated as it is for application in Section 6.

\noindent Here is a further extensions of Lemma \ref{lem2} onto the case where the  initial
locations of $k$ workers may differ, in general. 
\begin{lemma}\label{lem3}
Let $\delta \in (0,1)$. 
Assume that, in conditions of Lemma \ref{lem2},
the initial locations of $k$ workers on the halfline are
$$
0<x=x_{-k+1}\le x_{-k+2}\le\ldots \le x_0
$$
where $x_0\le (2-\delta )x$.  Let ${\bf x} = (x_{-k+1},x_{-k+2},\ldots,x_0)$ and $P_k({\bf x})$ be the probability that, with this initial configuration, the dust particle at the origin will be never removed. 
Then
\begin{equation}\label{F3}
\lim_{x\to 0} \frac{\log P_k({\bf x})}{\log^2 x} = -\frac{k}{2\log 2}, 
\end{equation} 
where the convergence is uniform 
in $x_{-k+i}/x \in [1,(2-\delta))$.
\end{lemma}
\noindent The following useful corollary shows that the logarithmic tail asymptotics for the probability for all workers to escape to infinity without visiting the origin is mostly due to the corresponding dust and clock environments, and given that, the probability to escape is a power function of the initial value $x_0$.
 
\begin{cor} \label{corrr} In conditions of the previous Lemma \ref{lem3},  
let  $Z$ be the conditional (upon the dust environment) probability that the $k$ workers will never return to the origin. Then there exists a constant $M>0$ such that, for any small $\widehat{\delta}$ and
uniformly in $x_{-k+i}/x \in [1,(2-\delta))$, 
\begin{align*}
{\mathbf P} (Z> x_0^M) \ge 
e ^{-k (1+ \delta) \log^2(x_0) / 2\log 2}, 
\end{align*}
for all $x_0 $ small enough.
\end{cor}

\noindent
{\sc Proof of Lemma \ref{lem1}}. We obtain separately the upper and
the lower bounds that are logarithmically equivalent.
\vspace{0.2cm}

\noindent
Write, for short, $P(x)=P_1(x)$. Clearly, $P(x)$ is an increasing function
that 
tends to 1 if $x\to\infty$, and to 0 if $x\to 0$. 
By the total probability formula,
\begin{equation}\label{BF}
P(x) = {\bf P} (D(x)) = \int_0^{x}{\bf  P}(\psi\in dy) {\bf P} (D(x) \ | \ 
\psi=y) = \int_0^xe^{-y}P(x+y) dy
\end{equation}
where $\psi$ is the distance from $x$ to the first dust particle on the right.  
Indeed, for the event of interest to happen, there should be at least
one dust particle within $(x,2x)$.
Recall that, for $x\in (0,1)$, we have
\begin{equation}\label{ex}
x/2 < 1-e^{-x} < x.
\end{equation}
\noindent 
{\sc Upper bound.} 
By monotonicity of $P$ (see Appendix) and by \eqref{BF} and \eqref{ex},
$P(x) \le \min (1,x) P(2x)$.  Let 
$m\equiv m_x=\min\{n:  2^nx\ge 1/2\}$, then  $m=-\frac{\log x}{\log 2}+
O(1)$, as $x\to 0$. 
Using the induction argument, we get:
\begin{eqnarray*}
P(x) &\le & x \cdot P(2x) \le x\cdot 2x\cdot \ldots \cdot 2^{m-1}x \cdot P(2^mx) 
\le x^m 2^{m(m-1)/2} \\
&=& \exp \left( m \log x + (1+o(1))\log 2\cdot m(m-1)/2 \right)\\
&=& \exp \left(-(1+o(1))\frac{\log^2 x}{2\log 2} \right), 
\end{eqnarray*}
as $x\to 0$.
\vspace{0.2cm}

\noindent
{\sc Lower bound.} We again use the monotonicity property of $P$ and \eqref{BF} and \eqref{ex}. For any $0<x,\varepsilon < 1/2$.
\begin{eqnarray*}
P(x) &\ge &  \int_{x(1-2\varepsilon)}^x e^{-y} P(x+y) dy \\
&\ge & 2\varepsilon x e^{- 2x(1-\varepsilon)}P(2x(1-\varepsilon )).
\end{eqnarray*}
Let $\varepsilon_n = \frac{c}{n^{\gamma}}$ be a decreasing to 0 sequence,
where $0<c<1/2$ and $\gamma >2$.
Using consequently $\varepsilon_1,\varepsilon_2,\ldots$ in place of
$\varepsilon$, we get:
\begin{eqnarray*}
P(x) &\ge & 2\varepsilon_1x e^{-2x(1-\varepsilon_1)}
P(2x(1-\varepsilon_1))\\
&\ge &2 \varepsilon_1x e^{-2x(1-\varepsilon_1)} \cdot
4\varepsilon_2 x(1-\varepsilon_1) 
e^{-4x(1-\varepsilon_1)(1-\varepsilon_2)}
P(4x(1-\varepsilon_1)(1-\varepsilon_2)) \ge \ldots \\
&\ge & 
(\prod_1^m \varepsilon_i) x^m2^{\frac{m(m+1)}{2}} A_{m-1} 
\exp ( -x\sum_1^m2^iB_i ) P(2^mB_mx) \\
&\ge &
A\frac{c^m}{(m!)^{\gamma}} \exp (m\log x+ \log 2\cdot m(m-1)/2
-x2^{m+1})P(2^mBx),
\end{eqnarray*}
where $A_n=\prod_1^n (1-\varepsilon_i)^i$,
$A=\prod_1^{\infty} (1-\varepsilon_i)^i$, $B_n=\prod_1^n (1-\varepsilon_i)$
and $B=\prod_1^{\infty} (1-\varepsilon_i)$ are strictly positive numbers.
Letting again $m$ be the integer part of $\frac{|\log x|}{\log 2}$, we get, as $x\to 0$,  
$$
P(x) \ge  AP(B/2) e^{-2}\exp \left(-\frac{\log^2x}{2\log 2} (1+o(1))\right) = \exp \left(-\frac{\log^2x}{2\log 2} (1+o(1))\right)
$$
since $2^mx = O(1)$, $c^m=e^{O(|\log x|)}=e^{o(\log^2x)}$ and
$$
(m!)^{\gamma}=\exp (\gamma |\log x| \log |\log x| (1+o(1)))=
\exp (o(\log^2x)).
$$
\vspace{0.2cm}

\noindent
{\sc Proof of Lemma \ref{lem2}}.\\
{\sc Upper bound}.
For the event of interest to occur, we need to have at least $k$ points of
the Poisson dust process to be within $(x,2x)$ and then, by the monotonicity  property 2 (see Appendix),
$$
P_k(x) \le {\mathbf P} (\xi_k \le x) P_k(2x).
$$
Here $\xi_k$ is the distance from $x$ to the $k$th point of the Poisson dust process
on $(x,\infty)$, so $\xi_n$ has Gamma distribution with parameters
$(k,1)$. Let $p_k(x) = {\mathbf P}(\xi_k\le x)$.
Then, for all $x$,
$$
x^k/k! \ge p_k(x) \ge x^ke^{-x}/k! \ .
$$ 
Then we may use the monotonicity properties 1 and 2 (see Appendix) to get, with $m$ as before,
\begin{eqnarray*}
P_k(x)&\le & \frac{x^k}{k!}P_k(2k)\\ 
&\le & \ldots \le x^{km}2^{\frac{km(m-1)}{2}}/(k!)^m \\
&=& \exp \left( -\frac{k\log^2x}{2\log 2} (1+o(1))\right).
\end{eqnarray*}
\vspace{0.2cm}

\noindent
{\sc Lower bound.} We define $\varepsilon_n$ as in the proof of Lemma
\ref{lem1}. 
Let $D_k(x)$ be the event the interest and  
\begin{align*}
G(x,x+c) = \{\mbox{no particles of the dust process within
interval} \ (x,x+c)\}.
\end{align*}

\noindent
Introduce the following events, for $n=1,2,\ldots$:\\
(i) let
\begin{align*}
 E_n = \{\mbox{  there is no dust particles in the interval} \ 
(2^{n-1}B_{n-1}x, 2^nB_nx)\}, 
\end{align*}
the probability of this event is not smaller than
$\exp (-2^nB_nx)$;\\
(ii) let
\begin{align*}
H_n = 
\{\mbox{there are exactly}  \ k \ \mbox{dust particles in the interval} \ 
(2^{n}B_{n}x, 2^nB_{n-1}x) \},
\end{align*} 
the probability of this event is equal to
$x^kK_n^ke^{-K_nx}/k!$ where $K_n= 2^nB_{n-1}\varepsilon_n$;\\
(iii) let 
\begin{align*}
J_n = \{\mbox{the clock of each of the} \ k \ \mbox{ workers rings
exactly once within} \ k \  \\
\mbox{consecutive rings numbered} \ k(n-1)+1, k(n-1)+2,\ldots, k_n\}
\end{align*} 
(we number the rings in order of their appearance),  
the probability of this  
event is not smaller that the probability
$$
p:= {\mathbf P} (X_k \le Y_k)>0
$$
where $X_k$ is the maximum of $k$ i.i.d. exponential-1 r.v.'s, $Y_k$ 
the minimum of other $k$ i.i.d. exponential-1 r.v.'s,
and $X_k$ and $Y_k$ do not depend on each other.
\vspace{0.2cm}

\noindent
Taking $m$ as before, we get: 
\begin{eqnarray*}
P_k(x) &=& {\bf P} (D_k(x))\\
&\ge & {\bf P} \left(\cap_1^m (E_n\cap H_n\cap J_n)\right)
\cdot {\bf P}\left(D_k(2^mB_mx)
\cap G (2^mB_mx, 2^mB_{m-1}x)\right)\\
&\ge & e^{-x\sum_1^m2^nB_n} \cdot \frac{x^{km}2^{m(m-1)/2}B^{km}
\left(\prod_1^m\varepsilon_n\right)^ke^{-\sum_1^m 2^nB_{n-1}\varepsilon_nx}}{\left(k!\right)^m} \\
&\times & p^m  \left(P(B/2)-1+e^{-2^m\varepsilon_mx}\right).
\end{eqnarray*}
Here\\
(1) We apply the first monotonicity property from Appendix: 
we observe that after removing $mk$ particles,
all $k$ workers are located within the interval $(2^mB_mx, 2^mB_{m-1}x)$ and
there are dust particles only at the origin and to the right of this interval. Therefore, in we move all workers to the smallest point of this interval,
the probability of interest becomes smaller, so the first inequality follows.\\
(2) Given that, we have all $k$ points at $2^mB_mx$ and are interested in
the probability of excaping of all workers to infinity given there is no 
dust particles within $(2^mB_mx, 2^mB_{m-1}x)$. Then we apply basic
inequalities: for events $A$,$B$ of positive probabilities,
$${\mathbf P}(A \ | \ B) \ge {\mathbf P} (AB) \ge {\mathbf P}(A) -
{\mathbf P} (\overline{B}), \ \  \mbox{where} \ \ \overline{B} \ \ 
\mbox{is the complement of} \ \ B.
$$

\noindent
Similarly to that in the proof of the lower bound in the previous lemma, we have a number
of inequalities: 
\begin{eqnarray*}
&& \exp (-x\sum_1^m2^nB_n)  \ge  \exp (-2B)>0,\\
&& (k!)^m = \exp (O(|\log x| \log |\log x|) = \exp (o(\log x)^2),\\
&& B^{km} = \exp (O(|\log x|),\\
&& (\prod_1^m \varepsilon_n)^k = \exp (O(|\log x| \log |\log x|) = \exp (o(\log x)^2),\\ 
&& p^m=\exp (O(|\log x|))
\end{eqnarray*}
 and then 
\begin{align*}
x^{km}2^{km(m-1)/2} &=& \exp (-\log^2x/\log 2 + \log^2x/2\log 2 (1+o(1)))\\
&=& \exp (-(1+o(1))\log^2x/2\log 2)
\end{align*}
and, for any monotone function $h(n)\to\infty$, $h(n)=o(n)$,
\begin{align*}
\sum_1^m 2^nB_{n-1}\varepsilon_nx \le \sum_1^{h(m)} 2^nx +
\varepsilon_{h(m)} 2^{m+1}x \le 2^{h(m)+1}x + 2\varepsilon_{h(m)} \to 0.
\end{align*}

Since 
$$
e^{-2^m\varepsilon_mx}\ge e^{-\varepsilon_m/2}\to 1,
$$
the result follows. 
\vspace{0.2cm}

\noindent
{\sc Proof of Lemma \ref{lem3}}. 
The probability of interest is upper-bounded by $P_k(2x)$, so we may use the upper bound
from the previous lemma. The proof of the lower bound differs only in  the first stepl
where the event $E_1$ is replaced by a bigger event, which makes the
probability bigger.

{\sc Proof of Corollary \ref{corrr}} follows from the proofs of the lemmas \ref{lem2} and \ref{lem3}, since the events $G$, $E_n$ and $H_n$ relate to the dust environment and events $J_n$ to the clock environment.

\section*{Appendix}

\subsection*{Monotonicity properties}

Consider the following deterministic model. 
Assume the state space is halfline $[0,\infty)$ and there are $k$ workers there,
located initially at points $0 < x_{-k+1}\le x_{-k+2}\le \ldots \le x_0$. Let
${\bf x} = (x_{-k+1},\ldots,x_0)$, with worker $1$ located at point $x_{-k+1}$, worker $2$ at point $x_{-k+2}$, etc., worker $k$ at point $x_0$. 
\vspace{0.2cm}

\noindent
Assume next that there are infinitely
many dust particles located at points $0=d_0<d_1\le d_2 \le \ldots$ where
$d_1>x_0$, so all workers are located initially between $0$ and $d_1$.
Let ${\bf d} = (d_1,d_2,\ldots)$ be an infinite vector (that does not include $d_0=0$). 
\vspace{0.2cm}

\noindent
Assume further that there is given a fixed (predefined) order of moves of workers 
$\{w_n\}_{n\ge 1}$ that says that, first, worker $w_1$ jumps to its closest
dust particle and removes it, then worker $w_2$ jumps and removes its
closest particle etc. If a worker finds two particles at the same location,
it chooses any of them. If a worker finds that he is located, say, at point $y$
and there are two closest particles at points $0$ and $2y$, then it chooses
point $0$. We assume that each number $1,2,\ldots,k$ appears in the sequence
$\{w_n\}$ infinitely often. 
\vspace{0.2cm}

\noindent
Assume it takes a unit of time per move.  
Let $D({\bf x}, {\bf d})$ represent the event that all workers escape
to infinity without visiting $0$ (in other words, that by time $\infty$
all dust particles but the one at the origin have been removed).
\vspace{0.2cm}

\noindent
Then we have the following elementary monotonicity properties.
\vspace{0.2cm}

\noindent
{\bf Monotonicity property 1.} Let $\widehat{\bf x}=(\widehat{x}_{-k+1},
\ldots,\widehat{x}_0)$ be another initial location of $k$ workers.
Assume that $\widehat{{\bf x}}\ge {\bf x}$ component-wise and that
$\widehat{x}_0<d_1$. Then if event 
$
D({\bf x},{\bf d})$
occurs, then event $ D(\widehat{\bf x},{\bf d}),
$ occurs too. 
\vspace{0.2cm}

\noindent
Indeed, when any worker, say, $j$, moves first time, it finds that it is more
likely to move to the right from location
$\widehat{x}_{-k+j}$ than from location $x_{-k+j}\le \widehat{x}_{-k+j}$.
\vspace{0.2cm}

\noindent
{\bf Monotonicity property 2.} Let $c>0$. Let $\widetilde{x}_{-k+j}=x_{-k+j}+c$ and $\widetilde{d}_i=d_i+c$, for all $1\le j \le k$ and all $i\ge 1$.
Then if event 
$
D(\widetilde{\bf x},\widetilde{\bf d})$ occurs, then event $D({\bf x}, {\bf d})
$ occurs too.
\vspace{0.2cm}

\noindent
This property can be easily verified step by step.  
As a corollary, we have the following. Let $P_k({\bf x})$ be defined as in the previous Section. Then
$$
P_k({\bf x})\le P_k({\bf y}), \quad \mbox{for all} \quad 0<{\bf x}\le {\bf y}
\ \mbox{and} \ k=1,2,\ldots.
$$

\noindent
{\bf Acknowledgement}. The authors thank Takis Konstantopoulos for his
valuable comments.

\end{document}